\documentclass[smallextended,envcountsect]{svjour3}       
\usepackage{latexsym,amsmath,amsfonts,amscd,amssymb,verbatim,array,mathdots}
\smartqed  % flush right qed marks, e.g. at end of proof
\usepackage{chngcntr}
\counterwithin{equation}{section}

\DeclareMathOperator{\inte}{int}
\DeclareMathOperator{\gph}{gph}

\DeclareMathOperator{\Sol}{SOL}

\DeclareMathOperator{\CP}{CP}

\DeclareMathOperator{\A}{\mathcal{A}}
\DeclareMathOperator{\B}{\mathcal{B}}
\DeclareMathOperator{\Bo}{\mathbb{B}}

\DeclareMathOperator{\Hd}{\mathcal{H}}
\DeclareMathOperator{\Po}{\mathcal{P}}

\DeclareMathOperator{\Ri}{\mathcal{R}}

\DeclareMathOperator{\Oo}{\mathcal{O}}

\DeclareMathOperator{\Sa}{\mathbb{S}}
\DeclareMathOperator{\Sc}{\mathcal{S}}
\DeclareMathOperator{\VI}{VI}

\DeclareMathOperator{\rank}{rank}
\DeclareMathOperator{\R}{\mathbb{R}}

\begin{document}
	
\title{Solution maps of polynomial variational inequalities}
	
\author{Vu Trung Hieu}

\institute{Vu Trung Hieu \at Sorbonne Universit\'e, \textsc{CNRS}, \textsc{LIP6}, F-75005, Paris, France\\
	E-mail: trung-hieu.vu@lip6.fr}

\date{In memory of my teacher Vu Thi Hai (1968-1991) \\ \\
Received: date / Accepted: date}
	
\maketitle
	
	%\medskip
\begin{abstract}
In this paper, we investigate several properties of the solution maps of variational inequalities with polynomial data. First, we prove some facts on the $R_0$-property, the local boundedness, and the upper semicontinuity of the solution maps. Second, we establish results on the solution existence and the local upper-H\"{o}lder stability under the copositivity condition. Third, when the constraint set is semi-algebraic, we discuss the genericity of the $R_0$-property and the finite-valuedness of the solution maps. 
\end{abstract}

	
	
\section{Introduction} 

Polynomial complementarity problems and polynomial variational inequalities have been recently investigated by many authors, see, for example 
\cite{Gowda16,HYY2015a,LLP2018,SongQi2015} and the references therein. These problems are natural extensions of affine variational inequalities (see \cite{LTY2005}  and the references therein) and special cases of weakly homogeneous variational inequalities introduced by Gowda and Sossa \cite{GoSo18}.

In the recent past, several authors studied different properties of polynomial complementarity problems, tensor complementarity problems and tensor variational inequalities, which are subclasses of polynomial variational inequalities. In particular, the solvability, the global uniqueness, and the boundedness of the solution sets have been studied in \cite{BHW2016,Gowda16,LHL2017,SongYu2016,WHQ18}. In \cite{Hieu18}, the author discussed the upper semicontinuity and the finite-valuedness of the solution maps of tensor complementarity problems. 
Very recently, the authors \cite{LLP2018} investigated the genericity and the H\"{o}lder stability for semi-algebraic variational inequalities where the related maps are the sums of semi-algebraic maps and parametric vectors. 

In this paper, we investigate several properties of the solution maps of variational inequalities with polynomial data. First,  we introduce the $R_0$-property and show that it plays an important role in the investigation. Since polynomial maps are weakly homogeneous, the normalization argument (see, e.g. \cite{AusTeb2003,GoSo18,OettliYen95}) can be applied. Several facts on the local boundedness and the upper semicontinuity of the solution maps are shown. Second, we obtain a result on the existence of solutions under the copositivity. We develop the results on solution stability of copositive  linear complementarity problems and affine variational inequalities in \cite{CPS1992,LTY2005} for copositive polynomial variational inequalities. 
Third, when the constraints are polynomial, techniques from semi-algebraic geometry (see, e.g. \cite{BCF98}) and differential geometry (see, e.g. \cite{Loring_2010})  can be used. Under some mild
conditions of the constraint set, we prove the genericity of the $R_0$-property and the finite-valuedness of the solution maps.

The present paper is organized as follows: In the next section, we give a short introduction to variational inequalities, asymptotic cones, polynomial maps, and semi-algebraic sets. In Section \ref{sec:usc}, we investigate the $R_0$-property and the upper semicontinuity of the solution maps. In Section \ref{sec:exist}, we prove some facts on the solution existence and stability of copositive polynomial variational inequalities. In the last section, two results on genericity are shown.

\section{Preliminaries}%\label{sec:pre}

This section gives a short introduction to variational inequalities, asymptotic cones, polynomial maps, and semi-algebraic sets.

\subsection{Variational inequalities and asymptotic cones}
The usual scalar product of two vectors $x, y\in\R^n$ is denoted by $\langle x,y\rangle$. Let $K$ be a nonempty closed convex subset of $\R^n$ and $F: \R^n\to\R^n$ be a continuous vector-valued map. The \textit{variational inequality} defined by $K$ and $F$ is the problem:
$$
{\rm Find} \  x\in K\ {\rm such\ that}\ \langle
F(x),y-x\rangle\geq 0,\  \forall y\in K.
$$
The problem and the corresponding solution set are denoted by $\VI(K,F)$ and $\Sol(K,F)$, respectively.  
By the continuity of $F$ and the closedness of $K$, it is not difficult to check that $\Sol(K,F)$ is closed.

If $F$ is a polynomial map then $\VI(K,F)$ is called a \textit{semi-polynomial variational inequality}. Furthermore, if $K$ is defined by finitely many polynomial equations and inequalities, then the problem is a \textit{polynomial variational inequality}.

\begin{theorem}\label{thm:HS}{\rm(The Hartman-Stampacchia Theorem, \cite[Chapter 1, Theorem 3.1]{KindStam})}
	If $K$ is compact, then the solution set $\Sol(K,F)$ is nonempty.
\end{theorem}

Let us recall that a nonempty set $C\subset \R^m$ is called a \textit{cone} if $\lambda>0$ and $x\in C$ then $\lambda x\in C$. The cone $C$ is bounded if and only if $C=\{0\}$. Note that $C$ is a cone if and only if $\R^m\setminus C$ is a cone. We denote by $\inte C$ and $C^*$ the interior and the  dual cone of $C$, respectively.

When (the closed convex set) $K$ is a cone, the \textit{complementarity problem} defined by $K$ and $F$, denoted by $\CP(K,F)$, is to find a vector $x\in\R^n$
satisfying the following conditions:
\begin{equation}\label{CP} x\in K, \ 
F(x)\in K^*, \ \langle
F(x),x\rangle=0.
\end{equation}
In this setting, it is known that a vector $x$ solves $\CP(K,F)$  if and only if $x$ solves  $\VI(K,F)$ \cite[Proposition 1.1.3]{FaPa03}. Therefore,  the solution set of $\CP(K,F)$ is also denoted by $\Sol(K,F)$.

\begin{remark} It is easy to see that a vector $x$ solves  $\CP(K,F)$ if and only if there is $\lambda\in\R^n$ such that
\begin{equation}\label{CP_sys1} x\in K, \ F(x)-\lambda=0,  \ \langle
	\lambda,x\rangle=0, \ \langle
	\lambda,y\rangle\geq 0 \  \forall y\in K.
	\end{equation}
\end{remark}

\begin{remark} If $K$ is a cone and $F$ is homogeneous of degree $d>0$, i.e. $F(tx)=t^{d}F(x)$ for all $t> 0$ and $x\in\R^n$, then the solution set of $\CP(K,F)$ contains $0$ and is a closed cone.
\end{remark}

The \textit{asymptotic cone} of $K$ is defined by 
$$
K^{\infty}=\left\lbrace v\in\R^n:\exists t_k\to \infty, \exists x_k\in K \text{ with } \lim_{k\to\infty} \frac{x_k}{t_k}=v\right\rbrace.
$$
By the convexity of $K$, $K^{\infty}$ coincides with  the recession cone  of $K$ which
is defined by the set of vectors $v \in\R^n$ such that for some vector
$x\in K$ the ray $\{x+tv: t\geq 0\}$ is contained in $K$ \cite[p.158]{FaPa03}. So, one has $K=K+K^{\infty}$. Recall that the cone $K^{\infty}$ is closed and convex;  $K$ is bounded if and only if $K^{\infty}=\{0\}$; if $K$ is a cone then $K^{\infty}=K$.

Let $P$ be a polynomial map, i.e. $P=(P_1,\dots,P_n)$ where $P_l$ is polynomial in $n$ variables, $l=1,\dots,n$. The maximum of the numbers $\deg P_l$ is called the \textit{degree} of the polynomial map $P$ and one denotes $\deg P=d$. 
We denote $P^{\infty}=(P_1^{\infty},\dots,P_m^{\infty})$, where $P_l^{\infty}$ is the homogeneous component of degree $d$ of $P_l$, $l=1,\dots,n$. Clearly, $P^{\infty}$ is the \textit{leading term} of the polynomial map $P$, i.e.
$$
P^{\infty}(x)=\lim_{\lambda\to+\infty}\frac{P(\lambda x)}{\lambda^d}, \ \forall x\in\R^n,
$$
and the map $P^{\infty}$ is homogeneous of degree $d$. 

\textit{Throughout the paper, we assume that $K$ is a nonempty, closed, and convex set, $P$ is a polynomial map of degree $d$, where $d$ is a positive integer, and let $\Sc:=\Sol(K^{\infty},P^{\infty})$}. 

Let $\Po_{d}$ be the linear space of all polynomial maps $Q=(Q_1,\dots,Q_n)$ of degree at most $d$, $m$ be the dimension of $\Po_{d}$, and $X$ be the vector consisting of all monomials degree at most $d$ which is listed by the lexicographic ordering \begin{equation}\label{monomials}
X := (1,x_1,x_2,\dots,x_n, x_1^2, x_1x_2,\dots,x_1x_n,\dots,x_1^d,x_1^{d-1}x_2,\dots,x_n^d)^T.
\end{equation}
For any polynomial map $Q\in\Po_d$, there exists a unique matrix  $A\in\R^{n\times m}$,
\begin{equation}\label{A}
A=\begin{bmatrix}
a_{11}& a_{12} & \cdots & a_{1m}  \\ 
a_{21}& a_{22} & \cdots & a_{2m}  \\ 
&  & \vdots &  \\ 
a_{n1}& a_{n2} & \cdots & a_{nm} 
\end{bmatrix},
\end{equation}
such that $Q(x)=AX$. The norm $\|Q\|$ of $Q$ is defined by the Frobenius norm of the coefficients of
matrix $A$. Let $\{Q^k\}$ be a convergent sequence in $\Po_d$ with $Q^k\to Q$ and $\{x^k\}$ be a convergent sequence in $\R^n$ with $x^k\to \bar x$. Then $Q^k(x^k)$ also is convergent with $Q^k(x^k)\to Q(\bar x)$.

\begin{remark}\label{P_infty}
Let $\{Q^k\}$ be a sequence in $\Po_{d}$ with $Q^k\to P$. Assume that $Q^k(x)=A^kX$ and $P(x)=BX$. Clearly, one has $A^k \to B$. It follows that $(Q^k)^{\infty}\to P^{\infty}$.
\end{remark}

The $R_0$-property of linear complementarity problems and affine variational inequalities has been investigated in \cite{JaTi87,OettliYen95} and \cite[p.189]{FaPa03}. We introduce a generalization of this property for semi-polynomial variational inequalities.

\begin{definition} One says that the problem $\VI(K,P)$ has the $R_0$-\textit{property} or $(K,P)$ is an \textit{$R_0$-pair} if the cone $\Sc$ is trivial, i.e., $\Sol(K^{\infty},P^{\infty})=\{0\}$.
\end{definition}

\begin{remark}\label{P_Q} Let $Q\in\Po_{d-1}$. Then $(P+Q)^{\infty}=P^{\infty}$. Furthermore, if $(K,P)$ is an $R_0$-pair, then $(K,P+Q)$ also is an $R_0$-pair. 
\end{remark}
\begin{remark}\label{K_compact} If $K$ is compact, then $K^{\infty}=\{0\}$, and hence $(K,P)$ is an $R_0$-pair.
\end{remark}

Let $\Ri_0(K,d)$ be the set of all polynomial maps $Q$ of degree $d$ such that $(K,Q)$ is an $R_0$-pair.

\begin{remark}The set $\Ri_0(K,d)$ is a cone. Indeed, for each $t>0$ and each polynomial map $Q$, one has 
$(tQ)^{\infty} =tQ^{\infty}$.
This implies that $$\Sol(K^{\infty},(tQ)^{\infty})=\Sol(K^{\infty},tQ^{\infty}).$$
 Moreover, it is easy to check that above sets coincide with $\Sol(K^{\infty},Q^{\infty})$, i.e.,
$$\Sol(K^{\infty},(tQ)^{\infty})=\Sol(K^{\infty},tQ^{\infty})=\Sol(K^{\infty},Q^{\infty}).$$
Therefore, $\Sol(K^{\infty},Q^{\infty})$ is bounded iff $\Sol(K^{\infty},(tQ)^{\infty})$ is bounded, for any $t>0$. This implies that $\Ri_0(K,d)$ is a cone in $\Po_{d}$.
\end{remark}

This paper mostly focuses on two solution maps $\Sol$ and $\Sol_P$ which are respectively defined by
\begin{equation}\label{Sol}
\Sol:\Po_{d}\rightrightarrows \R^n, \ Q\to \Sol(K,Q),
\end{equation}
and
\begin{equation}\label{Sol_P}
\Sol_{P}:\R^n\rightrightarrows\R^n, \ p\to \Sol(K,P+p).
\end{equation}

\begin{remark}\label{cl_graph}
By definition, it is easy to see that the map $\Sol$ is closed, i.e. the graph of $\Sol$, which is defined by
	$$\gph(\Sol)=\big\{(Q,x)\in \Po_{d}\times \R^n: x\in \Sol(K,Q)\big\},$$
	is closed in $\Po_{d}\times \R^n$. Similarly, the map $\Sol_{P}$ also is closed.
\end{remark}

\subsection{Semi-algebraic sets and the LICQ}

Recall that a set in $\R^n$ is \textit{semi-algebraic}, if it is the union of finitely many subsets of the form
\begin{equation*}\label{basicsemi}
\big\{x\in \R^n\,:\,f_1(x)=\dots=f_\ell(x)=0,\ g_{\ell+1}(x)<0,\dots,g_m(x)<0\big\},
\end{equation*}
where $\ell,m$ are natural numbers, and $f_1,\dots, f_\ell, g_{\ell+1},\dots,g_m$ are polynomials with real coefficients. For further details on semi-algebraic geometry, we refer to \cite{BCF98}.

Suppose that $S_1\subset \R^m$ and $S_2\subset \R^n$ are semi-algebraic sets. A vector-valued map $G:S_1\to S_2$ is said
to be semi-algebraic \cite[Definition 2.2.5]{BCF98}, if its graph is a semi-algebraic subset in $\R^m\times\R^n$.  

Let $S\subset \R^m$ be a semi-algebraic set. Then there exists a decomposition of $S$ into a disjoint union of semi-algebraic subsets \cite[Theorem 2.3.6]{BCF98} $$S=\bigcup_{i=1}^sS_i,$$ where each $S_i$ is semi-algebraically homeomorphic to $(0,1)^{d_i}$, i.e., there is a map $h:S_i \to (0,1)^{d_i}$ such that $h$ is semi-algebraic and homeomorphic. Let $(0,1)^{0}$ be a point and  $(0,1)^{d_i}\subset \R^{d_i}$ be the set of points $x=(x_1,\dots,x_{d_i})$ such that $x_j\in (0,1)$ for all $j=1,\dots,d_i$. The \textit{dimension} of $S$ is defined by
$$\dim S:=\max\{d_1,\dots,d_s\}.$$		
The dimension is well-defined and does not depend on the decomposition of $S$. 

If the semi-algebraic set $S$ is nonempty and $\dim S=0$, then  $S$ has finitely many points. If $\dim(\R^m\setminus S)<m$, then $S$ is \textit{generic} in $\R^m$ in the sense that $S$ contains a countable intersection of dense and open sets in $\R^m$.

Let $X\subset \R^m$ and $Y\subset \R^n$ be manifolds.
The tangent spaces of $X$ at $x$ and of $Y$ at $y$ are denoted by $T_xX$ and $T_yY$, respectively. Consider the differentiable map $\Phi:X \to Y$. A point $y \in \R^n$ is called a \textit{regular value} of $\Phi$ if either the level set $\Phi^{-1}(y)$ is empty or the derivative
map 
$$D\Phi(x): T_xX\to T_yY$$
is surjective at every point $x \in \Phi^{-1}(y)$. So, $y$ is a regular value of $\Phi$ if and only if $\rank D\Phi(x)=n$ for all $x\in \Phi^{-1}(y)$.

\begin{remark}%\label{RegularLevel}
	Consider a differentiable semi-algebraic map $\Phi:X \to \R^n$, where $X\subset \R^n $. Assume that $y \in \R^n$ is a regular value of $\Phi$ and $\Phi^{-1}(y)$ is nonempty. Applying the regular level set theorem \cite[Theorem 9.9]{Loring_2010}, one has  $\dim\Phi^{-1}(y)=0$; this implies that $\Phi^{-1}(y)$ has finitely many points.
\end{remark}

\begin{remark}%\label{Sard_para}
	Let $\Phi : \R^m\times X\to \R^n$ be a differentiable semi-algebraic map, where $X\subset \R^n $. Assume that $y \in \R^n$ is a regular value of $\Phi$. The Sard theorem with parameter \cite[Theorem 2.4]{DHP16} says that there is a generic
	semi-algebraic set $\Sa\subset\R^m$ such that, for every $p\in \Sa$, $y$ is a regular value of the map
	$\Phi_p:X\to \R^n$ with $x \mapsto \Phi(p,x).$
\end{remark}

%\subsection{Linearly independent constraint qualification}

Suppose that the constraint $K$ is represented by finitely many convex polynomial functions $g_i(x),i\in I,$ and finitely many affine functions $h_j(x),j\in J,$ as follows:
\begin{equation}\label{K}
K = \left\lbrace x \in \R^n : g_i(x) \leq 0, i\in I, \; h_j(x) = 0, j \in J\right\rbrace.
\end{equation}
For each index set $\alpha \subset I=\{1,\dots,s\}$, the \textit{pseudo-face} $K_{\alpha}$ of $K$  is defined by 
$$
K_{\alpha}=\left\lbrace x\in \R^n:g_i(x) =0, \forall i\in\alpha, g_i(x) <0, \forall i\notin\alpha, h_j(x) = 0,  \forall j\in J\right\rbrace.
$$
The number of pseudo-faces of $K$ is finite and these pseudo-faces  establish a disjoint decomposition of $K$. So, we obtain
\begin{equation}\label{Sol_decom}
\Sol(K,F)=\displaystyle\bigcup_{\alpha\subset I}\left[ \Sol(K,F)\cap K_{\alpha}\right].
\end{equation}

For each $x\in K$, denote by $I(x)$ the active index set at $x$ which is defined by
$I(x)= \{i\in I : g_i(x) = 0\}.$  
One says that $K$ satisfies the \textit{linearly independent constraint qualification} ($\rm LICQ$ for short), if the gradient vectors $$\{\nabla g_i(x), \nabla h_j(x), i\in I(x),j\in J\}$$ are linearly independent for all point $x\in K$. If the $\rm LICQ$ holds on $K$, then the Abadie constraint qualification (see, e.g. \cite[p.~17]{FaPa03}) also holds on $K$. Hence, the Karush-Kuhn-Tucker conditions can be applied (see, e.g. \cite[Proposition 1.3.4]{FaPa03}).

\section{Upper semicontinuity of solution maps}\label{sec:usc}

This section focuses on upper semicontinuity of the solution maps $\Sol$ and $\Sol_P$ given by \eqref{Sol} and \eqref{Sol_P}, respectively.  A close relation between the upper semicontinuity and the $R_0$-property is shown.

\subsection{Local boundedness}

The following theorem describes a relation between the $R_0$-property and the boundedness of certain solution sets.

\begin{proposition}\label{bounded1} Consider the following statements:
	\begin{description}
		\item[\rm(a)] $(K,P)$ is an $R_0$-pair;
		\item[\rm(b)] $\Sol(K,P+Q)$ is bounded, for every $Q\in\Po_{d-1}$;
		\item[\rm(c)] For any bounded open set  $\Oo\subset \Po_{d-1}$, the following set is bounded:
		$$S_{\Oo}:=\bigcup_{Q\in \Oo} \Sol(K,P+Q).$$
	\end{description}
	One has $ \rm(a) \Rightarrow \rm(c) \Rightarrow \rm(b)$. Moreover,  if $K$ is a cone then the three statements are equivalent.
\end{proposition}
\begin{proof} $\rm(a) \Rightarrow \rm(c)$ Assume that $(K,P)$ is an $R_0$-pair and, on the contrary,  there is a bounded open set $\Oo$ such that $S_{\Oo}$ is unbounded. There exists an unbounded sequence $\{x^k\}\subset K$ and a sequence 
	$\{Q^k\}\subset\Oo$
	such that $x^k\in\Sol(K,P+Q^k)$ for every $k$. By the unboundedness of $\{x^k\}$, without loss of generality, we can assume that $\|x^k\|^{-1}x^k\to\bar x $ with $\|\bar x\|=1.$ Because of the boundedness of $\{Q^k\}$, we can suppose that $Q^k\to \overline Q$ with $\deg \overline Q\leq d-1$. 
	
	By assumptions, one has 
	\begin{equation}\label{VI_k}
	\left\langle (P+Q^k)(x^k),y-x^k \right\rangle\geq 0, \ \forall y\in K.
	\end{equation}
	Let $u\in K$ be fixed. Then, for every $v\in K^{\infty}$, we have $u+\|x^k\|v\in K$ for any $k$. From \eqref{VI_k}, we deduce that
	$$\left\langle (P+Q^k)(x^k),u+\|x^k\|v-x^k \right\rangle\geq 0.$$
	Dividing this inequality by $\|x^k\|^{d+1}$ and letting $k\to+\infty$, we obtain
	$$\left\langle (P+\overline Q)^{\infty}(\bar x),v-\bar x \right\rangle=\left\langle P^{\infty}(\bar x),v-\bar x \right\rangle\geq 0,$$
and hence it follows that $\bar x \in \Sc=\{0\}$. As $\|\bar x\|=1$, this is a contradiction; therefore, $S_{\Oo}$ must be bounded. The assertion $\rm(c)$ is proved.	
	
	$\rm(c) \Rightarrow \rm(b)$ Assume that $\rm(c)$ holds, but there is $\overline Q\in\Po_{d-1}$ such that $\Sol(K,P+\overline Q)$ is unbounded. Then there exists a bounded and open set $\Oo$ such that $\overline Q\in \Oo$. Clearly, 
	$$\Sol(K,P+\overline Q)\subset S_{\Oo}.$$
This is impossible, because $S_{\Oo}$ is bounded; hence $\rm(b)$ follows. 
	
	Assuming that $K$ is a cone, we need only to prove that $\rm(b) \Rightarrow \rm(a)$. Consider $Q:=P-P^{\infty}\in\Po_{d-1}$. Then the assertion $\rm (b)$ implies that $\Sol(K,P^{\infty})$ is a bounded cone; consequently, we get  $\Sol(K,P^{\infty})=\{0\}$. The proof is complete. \qed
\end{proof}

The following lemma will be used in the proof of Theorem \ref{bounded2}.

\begin{lemma}\label{open_cone} The cone $\Ri_0(K,d)$ is open in $\Po_d$.
\end{lemma}
\begin{proof} We need only to prove that $\Po_d\setminus\Ri_0(K,d)$ is closed. Let $\{Q^k\}$ be a sequence in $\Po_d\setminus\Ri_0(K,d)$ such that $Q^k\to Q$. Then $(Q^k)^{\infty}\to  Q^{\infty}$ by Remark \ref{P_infty}. Moreover, for each $k$, $\Sol(K^{\infty},(Q^k)^{\infty})$ is unbounded. Hence, there exists an unbounded sequence  $\{x^k\}$ such that, for each $k$, $x^k\in\Sol(K^{\infty},(Q^k)^{\infty})$. Without loss of generality we can assume that  $x^k\neq 0$ for all $k$ and 
	$\|x^k\|^{-1}x^k\to\bar x$ with $\|\bar x\|=1.$
	By \eqref{CP_sys1}, for each $y\in K^{\infty}$, we get
	$$\langle
	(Q^k)^{\infty}(x^k),x^k\rangle=0, \ \langle
	(Q^k)^{\infty}(x^k),y\rangle\geq 0.$$
	Dividing the above equation and inequality by, respectively, $\|x^{k}\|^{d_k+1}$ and $\|x^{k}\|^{d_k}$, where $d_k$ is the degree of $Q^k$, and letting $k\to+\infty$, one gets
	$$\langle
	Q^{\infty}(\bar x),\bar x\rangle=0, \ \langle
	Q^{\infty}(\bar x),y\rangle\geq 0.$$
	This leads to $\bar x \in \Sol(K^{\infty}, Q^{\infty})$.
	As $\|\bar x\|=1$, we have $\bar x\neq 0$. It follows that $Q$ belongs to $\Po_d\setminus\Ri_0(K,d)$. The proof is complete. \qed
\end{proof}

Denote by $\B(0,\varepsilon)$  the open ball in $\Po_{d}$ with center at $0$ and radius $\varepsilon$. The closure of this ball is denoted by $\overline\B(0,\varepsilon)$.

The following theorem establishes a result on the local boundedness of the solution map $\Sol$ given in \eqref{Sol}.

\begin{theorem}\label{bounded2} If $(K,P)$ is an $R_0$-pair, then the map $\Sol$ is locally bounded at $P$, i.e. there exists $\varepsilon>0$ such that the set
	$$O_{\varepsilon}:=\bigcup_{Q\in \B(0,\varepsilon)} \Sol(K,P+Q)$$
	is bounded. Consequently, $\Sol(K,P+Q)$ is bounded for every $Q\in \B(0,\varepsilon)$.
\end{theorem}
\begin{proof} According to Lemma \ref{open_cone}, the cone $\Ri_0(K,d)$ is open in $\Po_d$. Then there is some $\varepsilon$ small enough such that 
	\begin{equation}\label{PB}
	P+\overline\B(0,\varepsilon)\subset \Ri_0(K,d).
	\end{equation}
	Suppose on the contrary, $O_{\varepsilon}$ is unbounded. Then there exists an unbounded sequence $\{x^k\}$ and a sequence $\{Q^k\}\subset\B(0,\varepsilon)$ such that $x^k\in\Sol(K,P+Q^k)$, $x^k\neq 0$ for every $k$, and
	$\|x^k\|^{-1}x^k\to\bar x$ with $\|\bar x\|=1$.
	
	By the compactness of $\overline\B(0,\varepsilon)$,  without loss of generality, we can assume that $Q^k\to Q$. Clearly, $P+Q^k\to P+ Q$ and
	\begin{equation}\label{PQ}
	P+ Q \in P+\overline\B(0,\varepsilon).
	\end{equation}
	By repeating the argument of the proof of Proposition \ref{bounded1}, we can show that $\bar x\in\Sol\left(K^{\infty}, (P+Q)^{\infty}\right)$. From  \eqref{PB} and \eqref{PQ}, $(K, P+Q)$ is an $R_0$-pair. This gives $\bar x=0$ which  contradicts $\|\bar x\|=1$. Therefore, $O_{\varepsilon}$ is bounded. 
	
	The proof is complete. \qed
\end{proof}

\subsection{Upper semicontinuity}

A set-valued map $\Psi:X\rightrightarrows Y$ between two topological spaces $X,Y$ is \textit{upper semicontinuous} at $x\in X$ iff for any open set $V\subset Y$ such that $\Psi(x)\subset V$ there is a neighborhood $U$ of $x$ such that $\Psi(x')\subset V$ for all $x'\in U$. If $\Psi$ upper semicontinuous at every $x\in X$ then one says that $\Psi$ is  upper semicontinuous on $X$. Recall that if $\Psi$ is closed, i.e. its graph is a closed set in $X\times Y$, and locally bounded at $x$ then $\Psi$ is upper semicontinuous at $x$ (see, e.g., \cite[p.139]{FaPa03}).

\begin{proposition}\label{usc_1} If $(K,P)$ is an $R_0$-pair and $\Sol(K,P)\neq\emptyset$, then the map $\Sol$ is upper semicontinuous at $P$.
\end{proposition}
\begin{proof} Assume that $(K,P)$ is an $R_0$-pair and $\Sol(K,P)\neq\emptyset$. From Remark \ref{cl_graph} and Theorem \ref{bounded2}, the map $\Sol$ is closed and locally bounded at $P$. Hence, $\Sol$ is upper semicontinuous at $P$. \qed
\end{proof}
\begin{corollary}%\label{usc_2} 
	If $K$ is compact, then the solution map $\Sol$	is upper semicontinuous on $\Po_{d}$.
\end{corollary}
\begin{proof} Suppose that $K$ is compact. Let $Q\in\Po_{d}$. From Remark \ref{K_compact}, $\VI(K,Q)$ has the $R_0$-property. Besides, Theorem \ref{thm:HS} says that $\Sol(K,Q)$ is nonempty. According to Proposition \ref{usc_1}, $\Sol$	is upper semicontinuous at $Q$. \qed
\end{proof}

\begin{corollary}\label{usc_3} Assume that $(K,P)$ is an $R_0$-pair and $p\in\R^n$. If the set $\Sol(K,P+p)$ is nonempty, then the solution map $\Sol_P$ is upper semicontinuous at $p$.
\end{corollary}
\begin{proof} Let $p\in\R^n$ be given. Clearly, $(K,P+p)$  is an $R_0$-pair. 	According to Proposition \ref{usc_1}, $\Sol$ is upper semicontinuous at $P+p$. Then for any open set $V$ containing $\Sol(K,P+p)$, there exists an open ball $\B(0,\varepsilon)$ such that 
	$$V \supset\bigcup_{Q\in \B(0,\varepsilon)} \Sol(K,(P+p)+Q)\supset\bigcup_{\|q\|<\varepsilon} \Sol(K,(P+p)+q).$$ 
	So, $U:=\{p+q\in\R^n:\|q\|<\varepsilon\}$ is an open neighbourhood of $p$ in $\R^n$ and $\Sol_{P}(U)\subset V$. It follows that $\Sol_{P}$ is upper semicontinuous at $p$. \qed
\end{proof}

The following theorem gives a sufficient condition for the $R_0$-property. The proof is a modification of one in \cite[Theorem 18.1]{LTY2005}.

\begin{theorem}%\label{bounded_usc} 
	Assume that $K$ is a cone. If there exists $Q\in\Po_{d-1}$ such that the following two conditions are satisfied:
	\begin{description}
		\item[\rm(a)] $\Sol(K,P+Q)$ is nonempty and bounded;
		\item[\rm(b)]  The solution map $\Sol$ is upper semicontinuous at $P+Q$;
	\end{description}	
	then $(K,P)$ is an $R_0$-pair.
\end{theorem}
\begin{proof} Since $K$ is a cone, we have $K=K^{\infty}$. Suppose that there is $Q\in\Po_{d-1}$ such that $\rm(a)$ and $\rm(b)$ hold, but $(K,P)$ is not an $R_0$-pair. Let $0\neq z\in\Sol(K,P^{\infty})$. From \eqref{CP_sys1}, there exists $\lambda\in\R^n$ such that 
	\begin{equation*}\label{KKT_H0}
	P^{\infty}(z)-\lambda=0,\
	\langle \lambda,z \rangle=0,\ 	\langle \lambda,y\rangle\geq 0 \ \;  \forall y\in K.
	\end{equation*}
	For each $t\in(0,1)$, we take $z_t:=t^{-1}z$ and $ \lambda_t:=t^{-d}\lambda$. We prove the existence of $Q_t\in\Po_d$, with $Q_t\to P+Q$ as $t\to 0$, satisfying
	\begin{equation}\label{KKT_Ht}
	Q_t (z_t)-\lambda_t=0,\
	\langle \lambda_t,z_t \rangle=0,\ \langle\lambda_t, y \rangle \geq 0 \  \forall y\in K.
	\end{equation}
	
	Suppose that $$P=P^{\infty}+P^{d-1}+\dots+P^{1}+P^{0}$$ and $$Q= Q^{d-1}+\dots+Q^{1}+ Q^{0},$$ where $P^k, Q^k$ are homogeneous polynomial maps of degree $k$ $(k=1,\dots,d-1)$ and $P^{0}\in\R^n,Q^{0}\in\R^n$. The sum $P+Q$ can be written as
	\begin{equation}\label{PQ_bar}
	P+Q=P^{\infty}+[P^{d-1}+Q^{d-1}]+\cdots+[P^{1}+ Q^{1}]+[P^{0}+Q^{0}].
	\end{equation}
Because $z=(z_1,\dots,z_n)$ is nonzero, there exists $l\in\{1,\dots,n\}$ such that $z_l\neq 0$; hence $z^{k}_l\neq 0$ for $k=1,\dots,d$. Take $\overline Q\in\Po_d$ with $$\overline Q(x)=\overline Q^d(x)+\cdots+\overline Q^1(x),$$
	where $\overline Q^k$ is a  homogeneous polynomial map of degree $k$ defined by
	$$\overline Q^k(x)=\left( a_{k1}x_l^{k},\dots,a_{kn}x_l^{k}\right), \ a_{ki}=-\frac{P_i^{k-1}(z)+Q_i^{k-1}(z)}{z_l^{k}}, \ i=1,\dots,n,$$
	with $P_i^{k-1}, Q_i^{k-1}$  are the $i-$th components of $P^{k-1}, Q^{k-1}$, respectively.
	It is easy to check that \begin{equation}\label{FQ0}
	P^{k-1}(z)+Q^{k-1}(z)+\overline Q^{k}(z)=0.
	\end{equation}
Choosing
	$Q_t=(P+Q)+t\overline Q$, we now prove that the system \eqref{KKT_Ht} is valid. Indeed, the last one in \eqref{KKT_Ht}   is obvious. 	The second one  in~\eqref{KKT_Ht} is obtained by
	$$\langle \lambda_t,z_t \rangle=\langle t^{-d}\lambda,t^{-1}z \rangle=t^{-d-1}\langle \lambda,z \rangle=0.$$
We now prove the first equality of \eqref{KKT_Ht}. From \eqref{PQ_bar}, we get
	$$ \begin{array}{cl}
	Q_t(z_t)-\lambda_t&=[(P+Q)+t\overline Q](t^{-1}z) - t^{-d}\lambda \medskip \\ 
	&= t^{-d}\left[P^{\infty} (z)-\lambda\right] +\sum_{k=1}^{d}t^{-(k-1)}\left[ P^{k-1} (z)+Q^{k-1} (z)+\overline Q^{k}(z)\right].
	\end{array} $$
This and \eqref{FQ0} imply that $Q_t(z_t)-\lambda_t=0$.	Hence, we get $z_t \in \Sol(K,Q_t)$. This holds for all $t\in (0,1)$.
	
	Since $\Sol(K,P+Q)$ is bounded, there is a bounded open set $V$ containing $\Sol(K,P+Q)$. By the upper semicontinuity of $\Sol$ at $P+Q$, there is $\varepsilon>0$ such that $\Sol(K,P+Q+Q')\subset V$ for all $Q'\in\Po_n$ and $\|Q'-(P+Q)\|<\varepsilon$. Taking $t$ small enough such that $\|Q_t-(P+Q)\|<\varepsilon$, we have $\Sol(K,Q_t)\subset V$. So, $z_t\in V$ for every $t>0$ sufficiently small. This is impossible, because $V$ is bounded and $z_t$ is unbounded as $t\to 0$. 
	
	The proof is complete. \qed
\end{proof}

\section{Solution existence and stability under copositivity condition}\label{sec:exist}

In this section, we will establish some results on solution existence and stability of semi-polynomial variational inequalities whose involved maps are copositive.

\subsection{Solution existence}

Let $C$ be a nonempty and closed subset of $\R^n$.  Note that
$q\in\inte C^*$ if and only if 
$\left\langle v,q \right\rangle >0$ for all $v\in C$ and $v\neq 0$ (see, e.g., \cite[Lemma 6.4]{LTY2005}).

Recall that the map $F$ is copositive on $K$
if $\left\langle F(x),x \right\rangle \geq 0$ for all $x\in K$, and  
monotone on $K$ if 
\begin{equation}\label{mono}
\left\langle F(y)-F(x),y-x\right\rangle \geq 0,\end{equation}
for all $x,y\in K$. 
If the inequality in \eqref{mono} is strict for all $y\neq x$, then $F$ is strictly monotone on $K$. If $0\in K$, $F(0)=0$, and $F$ is monotone on $K$, then $F$  is copositive on $K$.

%\subsection{Under the copositivity condition}

Theorem 6.2 in \cite{GoSo18} gives a result on the solution existence under a copositivity condition along with $\inte(K^*)\neq\emptyset$. We will now improve this result by weakening the interiority condition.

\begin{theorem}\label{cop1} Assume that $0\in K$ and $P$ is copositive on $K$. If $p\in\inte(\Sc^*)$, then $\Sol(K,P+p)$ is nonempty and bounded.
\end{theorem}
\begin{proof} If $K$ is compact then the assertion is obvious. Hence, we suppose that $K$ be unbounded. Let $p\in\inte(\Sc^*)$ be given. For each $k=1,2,\dots$, we denote $$K_k=\{x\in\R^n:x\in K,\|x\|\leq k\}.$$
	Clearly, the set $K_k$ is compact. Without loss of generality, we can assume that $K_k$ is nonempty. According to Theorem \ref{thm:HS},  $\VI(K_k,P+p)$ has a solution denoted by $x_k$. 
	
	We will show that the sequence $\{x^k\}$ is bounded. Suppose on the contrary that $\{x^k\}$ is unbounded with  $x^k\neq 0$, for all $k$, and 
	$\|x^k\|^{-1}x^k\to\bar x$. Clearly, one has $\bar x \in K^{\infty}$ and $\|\bar x\|=1$.
	For each $k$, it is true that
	\begin{equation}\label{VI_H}
	\left\langle P(x^k)+p,y-x^k \right\rangle \geq 0,
	\end{equation}
	for all $y\in K_k$. By fixing $y\in K_1$, dividing the inequality in \eqref{VI_H} by $\|x^k\|^{d+1}$ and letting $k\to+\infty$, we obtain
	$\left\langle P^{\infty}(\bar x),\bar x \right\rangle\leq 0$. Moreover, by the copositity of $P$, one has $\left\langle P(x^k),x^k \right\rangle\geq 0$. This leads to $\left\langle P^{\infty}(\bar x),\bar x \right\rangle \geq 0$. We thus get $\left\langle P^{\infty}(\bar x),\bar x \right\rangle= 0$.
	
	Let $v\in K^{\infty}\setminus \{0\}$ be fixed. For each $k$, we set
	$y_k:=\|x^k\|\|v\|^{-1}v.$
	Since $K=K+K^{\infty}$ and $0\in K$, one has $y_k\in K$ for any $k$. It is easy to see that $\|y^k\|=\|x^k\|\leq k$, hence that $y_k\in K_k$. 
	Now \eqref{VI_H} becomes
	$$\left\langle P(x^k)+p,\|x^k\|\|v\|^{-1}v-x^k \right\rangle\geq 0.$$
	Dividing this inequality by $\|x^k\|^{d+1}$ and taking $k\to+\infty$, we obtain
	$$\left\langle P^{\infty}(\bar x),v \right\rangle\geq \|v\|\left\langle P^{\infty}(\bar x),\bar x \right\rangle.$$
	From what has already been proved, we have
	$$\bar x\in K^{\infty}, \ \left\langle P^{\infty}(\bar x),\bar x \right\rangle = 0, \ \left\langle P^{\infty}(\bar x),v \right\rangle\geq 0 \ \forall v\in K^{\infty}.$$
	This means that $\bar x \in \Sc$. 
	
	Since $P$ is copositive, letting $y=0$ in \eqref{VI_H}, one has
	$$-\left\langle p,x^k\right\rangle \geq \left\langle P(x^k),x^k \right\rangle\geq 0.$$
	Dividing this inequality by $\|x^k\|$, as $k\to+\infty$, we get $\left\langle p,\bar x\right\rangle \leq 0$. 
This contradicts the assumption $p\in\inte(\Sc^*)$. Thus, the sequence $\{x^k\}$ must be bounded. 

We can assume that  $x^k\to \hat x$. We now prove that $\hat x$ solves $\VI(K,P+p)$. Indeed, for any $y\in K$, from \eqref{VI_H}, taking $k\to+\infty$, one has
	$$\left\langle P(\hat{x})+p,y-\hat{x} \right\rangle \geq 0.$$
	Hence, the nonemptiness of $\Sol(K,P+p)$ is proved.
	
	The boundedness of $\Sol(K,P+p)$ is proved by assuming that there exists an unbounded sequence of solutions $\{x_k\}\subset\Sol(K,P+p)$ with  $x^k\neq 0$, for all $k$, and 
	$\|x^k\|^{-1}x^k\to\bar x$ with $\|\bar x\|=1$.
Applying the normalization argument, we can show that $\bar x \in \Sc$ and  $\left\langle p,\bar x\right\rangle \leq 0$. This contradicts the assumption $p\in\inte(\Sc^*)$. 
	
	The proof is complete.
	\qed
\end{proof}

To illustrate Theorem \ref{cop1}, we provide the following example.

\begin{example} Consider the polynomial variational inequalities $\VI(K,P+p)$, where $K=\R^2_+$, $p=(p_1,p_2)^T\in\R^2$, and $P$ is given by
	$$P(x_1,x_2)=\begin{bmatrix}
	(x_1-x_2)^2\\
	(x_1-x_2)^2
	\end{bmatrix}.$$
	Clearly, one has $P^{\infty}=P,$ $K^{\infty}=K$, and  $P$ is copositive on $K$. 	Since $$\Sc=\{(x_1,x_2)\in\R^2_+: x_1-x_2=0\},$$
	one has
	$$\inte(\Sc^*)=\{(p_1,p_2)\in\R^2: p_1+ p_2>0\}.$$
	From Theorem \ref{cop1}, $\Sol(K,P+p)$ is nonempty and bounded for any $p\in \inte(\Sc^*)$. In fact, an easy computation shows that 
	\begin{equation*}\label{SolO}
	\Sol_{P}(p_1,p_2)=\left\{\begin{array}{ccc}
	L_0 & \text{ if } &p_1=0, \;p_2=0, \\
	L_{-p_1}&  \text{ if } & p_1=p_2, p_2<0,  \\
	\{(0,0)\}\cup\{(0,\sqrt{-p_2})\} &  \text{ if } & p_1> p_2,p_2<0, \\
	\{(0,0)\}\cup\{(\sqrt{-p_1},0)\} &  \text{ if } & p_1<0,p_1< p_2,\\	
	\{(0,0)\} & \text{ if } & \text{ otherwise,}
	\end{array}\right.
	\end{equation*}
	where $$L_{-p_1}=\{(x_1,x_2)\in\R^2_+: x_1-x_2=\sqrt{-p_1}\}.$$	
	Clearly, $	\Sol_{P}(p_1,p_2)$ is nonempty and bounded for any $p_1+ p_2>0$.
\end{example}

\subsection{Upper semicontinuity}

We now give a sufficient condition for the upper semicontinuity of $\Sol_{P}$ at $p$ under the copositivity.

\begin{proposition}\label{usc_int} Assume that $0\in K$ and $P$ is copositive on $K$. Then, $\Sol_{P}$ is upper semicontinuous on $\inte(\Sc^*)$.
\end{proposition}
\begin{proof}	Let $p$ be in $\inte(\Sc^*)$. Theorem \ref{cop1} says that $\Sol(K,P+p)\neq \emptyset$. Suppose that $\Sol_P$ is not upper semicontinuous at $p$. Then there exist a nonempty open set $V$ containing $\Sol(K,P+p)$, a sequence $\{p^k\}\subset \R^n$, and a sequence $\{x^k\}\subset K$ such that $p^k\to p$ and
	\begin{equation}\label{V_open1}
	x^k\in\Sol(K,P+p^k)\setminus V,
	\end{equation}
	for each $k$.	By repeating the argument of the proof of Theorem \ref{cop1}, one can prove that the sequence $\{x^k\}$ is bounded.	So, without loss of generality we can assume that  $x^k\to \bar x$. It is easy to check that  $\bar x\in\Sol(K,P+p)$, hence that $\bar x\in V$. Besides, since $V$ is open, the relation \eqref{V_open1} implies that $\bar x\notin V$. One obtains a contradiction. Therefore, $\Sol_P$ is upper semicontinuous at $p$. \qed
\end{proof}

\begin{corollary}\label{cor:copo} 
	Assume that $0\in K$ and $P$ is copositive on $K$. If $(K,P)$ is an $R_0$-pair, then the two following assertions hold:
	\begin{description}
		\item[\rm(a)] $\Sol_{P}(p)$ is nonempty and bounded for any $p\in\R^n$.
		\item[\rm(b)] $\Sol_{P}$ is upper semicontinuous on $\R^n$.
	\end{description}	
	
\end{corollary}
\begin{proof} Since $(K,P)$ is an $R_0$-pair, one has $\Sc=\{0\}$ and $\inte(\Sc^*)=\R^n$. From  Theorem \ref{cop1}, the assertion $\rm(a)$ follows.  From  Proposition \ref{usc_int}, $\Sol_{P}$ is upper semicontinuous at $p$, for any $p\in \R^n$. \qed
\end{proof}

\subsection{Local upper-H\"{o}lder stability}

In \cite{LTY2005}, the authors gave a result for the solution stability of copositive affine variational inequalities. In this section, we extend these results for copositive polynomial variational inequalities.

Let $p\in\R^n$ be given. If there exist $L>0,c>0$ and a neighborhood $U_{p}$ of $p$ such that 
$$\Sol_{P}(q)\subset \Sol_{P}(p)+L\|q-p\|^{c}\Bo(0,1) \ \ \forall q\in U_{p},$$
where $\Bo(0,1)$ is the closed unit ball in $\R^n$, then one says that $\Sol_{P}$ is \textit{locally upper-H\"{o}lder stable} at $p$.

When $K$ is semi-algebraic, the following result says that  the upper semicontinuity and the local upper-H\"{o}lder stability of $\Sol_{P}$ at $p$ are equivalent.

\begin{theorem}\label{thm:LLP}{\rm(see \cite{LLP2018})} Assume that $K$ is semi-algebraic and $\Sol_{P}(p)$ is nonempty. Then the map $\Sol_{P}$ is  upper semicontinuous at $p$ iff it is locally upper-H\"{o}lder stable at $p$.
\end{theorem}

\begin{proposition}\label{Holder} Assume that $K$ is semi-algebraic, $0\in K$, and $P$ is copositive on $K$. If $p\in\inte(\Sc^*)$, then $\Sol_{P}$ is locally upper-H\"{o}lder stable at $p$.
\end{proposition}
\begin{proof} Suppose that $p\in\inte(\Sc^*)$. Proposition \ref{usc_int} says that $\Sol_{P}$ is upper semicontinuous at $p$. According to Theorem \ref{thm:LLP}, $\Sol_{P}$ is locally upper-H\"{o}lder stable at $p$.
 \qed
\end{proof}

\begin{theorem}%\label{stability1} 
Assume that $K$ is semi-algebraic, $0\in K$, and $P$ is copositive on $K$. Let $p\in\inte(\Sc^*)$ be given. Then there exist constants $\varepsilon>0, L>0$ and $c>0$ with the following property: If $Q\in\Po_d$, $Q$ is copositive on $K$, and $q\in\R^n$ satisfy
	\begin{equation*}\label{eps}
	\max\{\|Q-P\|,\|q- p\|\}<\varepsilon,
	\end{equation*}
	then $\Sol(K,Q+q)$  is nonempty and bounded; and	
		\begin{equation}\label{ell}\Sol(K,Q+q)\subset \Sol(K,P+p)+L(\|Q-P\|+\|q-p\|)^{c}\Bo(0,1).	\end{equation}
\end{theorem}
\begin{proof}  We first prove that there exists $\delta>0$ such that if $Q\in \Po_d$, where $Q$ is copositive on $K$, and $ p\in\R^n$ with \begin{equation}\label{delta}
	\max\{\|Q-P\|,\|q- p\|\}<\delta,
	\end{equation}	 
then $\Sol(K,Q+q)$ is nonempty and bounded. Suppose that the assertion is false. Then there is a sequence $\{(Q^k,q^k)\}\subset \Po_{d}\times\R^n$ such that $(Q^k,q^k) \to (P,p)$, where $Q^k$ is copositive on $K$ for each $k$, and $\Sol(K,Q^k+q^k)$ is empty or unbounded. Due to Theorem \ref{cop1}, one has
 $$q^k\notin\inte(\Sol((Q^k)^{\infty},K^{\infty})^*).$$
 This means that there exists $x^k\in\Sol((Q^k)^{\infty},K^{\infty})$ satisfying $x^k\neq 0$ and $\left\langle x^k,q^k \right\rangle\leq 0$. We can assume that $\|x^k\|^{-1}x^k\to\bar x\in K^{\infty}$ with $\|\bar x\|=1$. 
 
It is not difficult to see that $\left\langle \bar x,p \right\rangle\leq 0$.
If we prove that $\bar x\in\Sc$, then this contradicts the assumption $p\in\inte(\Sc^*)$; and hence $\rm(a)$ will be proved. Thus, we only need to show that $\bar x\in \Sc$. From \eqref{CP}, one has
\begin{equation}\label{Qk}
\langle (Q^k)^{\infty}(x^k),x^k\rangle=0, \ \langle
(Q^k)^{\infty}(x^k),y\rangle\geq 0 \  \forall y\in K^{\infty}.
\end{equation}
It follows from Remark \ref{P_infty} that $(Q^k)^{\infty} \to P^{\infty}$. Let $y\in K^{\infty}$ be fixed; by dividing the equation and inequality in \eqref{Qk} by $\|x^k\|^{d+1}$ and $\|x^k\|^{d}$, respectively, and letting $k\to+\infty$, we obtain $$\langle P^{\infty}(\bar x),\bar x\rangle=0, \ \langle
 P^{\infty}(\bar x),y\rangle\geq 0.$$
As this holds for every $y\in K^{\infty}$, we get $\bar x\in\Sc$.
 
We now prove the inclusion \eqref{ell}. According to Proposition \ref{Holder}, there exist $L_0>0,c>0$ and $\varepsilon$ such that 
	\begin{equation}\label{ell0}
	\Sol(K,P+q)\subset \Sol(K,P+p)+L_0\|q-p\|^{c}\Bo(0,1)
	\end{equation}
	for all $q$ satisfying $\|q-p\|< \varepsilon$. 

Suppose $Q$ (copositive on $K$) and $q$ (in $\R^n$) satisfy \eqref{delta}. As $\Sol(K,Q+q)$ is nonempty, for each $z_q\in \Sol(K,Q+q)$, by setting 
	\begin{equation}\label{q_hat} \widehat q:=q+\left(Q- P\right)(z_q), \end{equation}
	we have
	$P(z_q)+ \widehat q=Q(z_q)+q$ and 
	$$\left\langle P(z_q)+ \widehat q, y-z_q\right\rangle=\left\langle Q(z_q)+q,y-z_q\right\rangle  \geq 0 \ \forall y\in K.$$ 
This gives \begin{equation}\label{z_q}
z_q\in\Sol(K,P+\widehat q).
\end{equation} 
Since  $\Sol(K,Q+q)$ is compact, there exists $\beta>0$ such that 
	\begin{equation}\label{norm_ineq}
	\|(Q-P)(z)\|\leq \beta\|Q-P\|
	\end{equation}
	for all $z\in\Sol(K,Q+q)$.  
	From \eqref{q_hat}, \eqref{norm_ineq}, and \eqref{delta},  we get
	$$\begin{array}{ll}
	\|\widehat q-p\|   &\leq  \; \|\widehat q-q\|+\|q-p\| \smallskip \\
	&\leq  \; \|(Q-P)(z_q)\|+\|q-p\| \smallskip \\ 
	&\leq  \; \beta \|Q-P\|+\|q-p\| \smallskip \\ 
	&\leq \; (1+\beta)\delta.
	\end{array}$$
	Choosing $\delta$ small enough such that $(1+\beta)\delta<\varepsilon$,  we have $\|\widehat q-p\|<\varepsilon$.	From \eqref{ell0}, \eqref{z_q}  and \eqref{norm_ineq}, there exists $x\in\Sol(K,P+p)$ such that
	$$\begin{array}{rl}
	\|z_q-x\| \; & \leq \; L_0\|\widehat q-p\|^{c} \smallskip\\
	&\leq  \; L_0\left(\|q-p\|+\beta \|Q-P\| \right)^{c} \smallskip    \\ 
	& \leq  \; L \left(\|q-p\|+\|Q-P\| \right)^{c},
	\end{array} $$
	where $L:=\max\left\lbrace L_0^{c}\beta,L_0^{c}\right\rbrace $. 
	
Since the inequality holds for any $z_q$ in $\Sol(K,Q+q)$, the  inclusion \eqref{ell} holds.  \qed
\end{proof}

\subsection{The GUS-property}

If $\VI(K,P+p)$ has a unique solution for every $p\in\R^n$, then $\VI(K,P)$ is said to have the globally uniquely solvable property (GUS-property).  The following theorem develops Theorem 4.3 in \cite{WHQ18} which concerning the GUS-property of tensor variational inequalities. %This also gives a partial answer to Question 5.2 in \cite{WHQ18}.

\begin{proposition}\label{GUS1} 
	Assume that $0 \in K$ and $P$ is strictly monotone on $K$. If $(K,P)$ is an $R_0$-pair, then $\VI(K,P)$ has the GUS-property and $\Sol_{P}$ is single-valued and continuous on $\R^n$.
\end{proposition}
\begin{proof} Suppose that $(K,P)$ is an $R_0$-pair. By the monotonicity of $P$, $Q(x):=P(x)-P(0)$ is monotone on $K$. Because $0$ belongs to $K$, the map $Q$ is copositive on $K$. From Remark \ref{P_Q},
	$(K,Q)$ is an $R_0$-pair. According to the assertion $\rm(a)$ in Corollary \ref{cor:copo}, the set $$\Sol(K,Q+q)=\Sol(K,P-P(0)+q)$$ 
	is nonempty for all $q\in\R^n$. This is equivalent to saying that $\Sol(K,P+p)\neq \emptyset$ for all $p\in\R^n$. 	
	Since $P$ is strictly monotone, so is $P+p$.  According to \cite[Theorem 2.3.3]{FaPa03}, $\VI(K,P+p)$ has at most one solution. Thus, $\VI(K,P+p)$ has a unique solution for every $p\in\R^n$. 
	
	Because $\VI(K,P)$ has the $R_0$-property, Corollary \ref{usc_3} says that $\Sol_{P}$ is upper semicontinuous on $\R^n$. Hence, the map is single-valued and continuous on $\R^n$.	
	The proof is complete. \qed
\end{proof}

To illustrate Proposition \ref{GUS1}, we give the following example.

\begin{example}%\label{example_2}
	Consider the polynomial variational inequality given by 
	$$K=\{x=(x_1,x_2)\in\R^2:x_1\geq 0\}, \ P(x)+p=\begin{bmatrix}
	x_1^3\\
	x_2^3
	\end{bmatrix}+\begin{bmatrix}
	p_1\\
	p_2
	\end{bmatrix},$$
	where $(p_1,p_2)\in\R^2$. It is easy to check that $0\in K$, $P$ is strictly monotone on $K$,  and $(K,P)$ is an $R_0$-pair. According to Proposition \ref{GUS1}, the problem has the GUS-property and $\Sol_{P}$ is single-valued and continuous on $\R^2$. In fact, an easy computation shows that
	$$\Sol_{P}(p_1,p_2)=\left\{\begin{array}{cl}
	\left\lbrace(\sqrt[3]{-p_1},\sqrt[3]{-p_2})\right\rbrace & \text{ if } p_1<0, \\
	\left\lbrace (0,\sqrt[3]{-p_2} )\right\rbrace & \text{ if } p_1\geq 0.
	\end{array}\right.$$
	This map is single-valued and continuous on $\R^2$. 
\end{example}

\section{Genericity}%\label{sec:gen}

In this section, we first prove the genericity of the $R_0$-property of polynomial variational inequalities under some mild
conditions. Then, we show that the solution map $\Sol$ is finite-valued on a generic semi-algebraic set of the parametric space. 

\subsection{Genericity of the $R_0$-property}

Let $\Hd_d$ be the vector space spanned by polynomial maps $H=(H_1,\dots,H_n)$, where all $H_l$ are homogeneous  of degree $d$. The dimension of $\Hd_d$ is denoted by $\rho$. Clearly, $\R^{n\times\rho}$ and  $\Hd_d$  are isomorphic. Let $X_d$ be a vector whose components are monomials of degree $d$ listed by lexicographic ordering
$$X_d=\left(x_1^d,x_1^{d-1}x_2,\dots,x_n^d\right)^T.$$
For any homogeneous polynomial map $H\in \Hd_d$, there is a unique $B\in \R^{n\times\rho}$,
$$B=\begin{bmatrix}
b_{11}& b_{12} &\cdots & b_{1\rho}  \\ 
b_{21}& b_{22}  &\cdots & b_{2\rho}  \\ 
\vdots & \vdots  & \ddots & \vdots \\ 
b_{n1}& b_{n2} &\cdots & b_{n\rho} 
\end{bmatrix},$$
such that $H(x)=BX_d$, where $$H_l(x)=b_{l1}x_1^d+b_{l2}x_1^{d-1}x_2+\dots+b_{l\rho}x_n^d \; (l=1,\dots,n).$$

Assume that $K$ is an unbounded polyhedral convex cone, which is the intersection of finitely many half-spaces containing the origin, given by   
\begin{equation}\label{K_0}
K = \left\lbrace x \in \R^n : Cx \leq 0\right\rbrace,
\end{equation}
where $C=(c_{ij})\in {\mathbb R}^{s\times n}$. The following lemma shows that the solution map of homogeneous polynomial complementarity problems,
$$\Gamma_K:\R^{n\times\rho}\rightrightarrows\R^n, \ B\mapsto \Gamma_K(B)= \Sol(K,BX_d),$$
is constant on a generic semi-algebraic set of $\R^{n\times\rho}$ provided that  the $\rm LICQ$ holds on $K$.

\begin{lemma}\label{generic_1} Assume that $K$ is a polyhedral convex cone given by \eqref{K_0} and the $\rm LICQ$ holds on $K$. Then there exists a  generic semi-algebraic set $\Sa\subset\R^{n\times\rho}$ such that $\Gamma_K(B)=\{0\}$ for any $B\in\Sa$.
\end{lemma}
\begin{proof} Firstly, 	since the $\rm LICQ$ holds on $K$, applying  \cite[Proposition 1.3.4]{FaPa03} for the $\VI(K,H)$,	we have $x\in\Sol(K,H)$ if and only if there exists $\lambda\in \R^{s}$ such that 
	\begin{equation}\label{KKT}
	\left\lbrace \begin{array}{l}
	H(x)+C^T\lambda=0,\\ 
	\lambda^T(Cx)=0, \; \lambda\geq 0, \; Cx \leq 0.
	\end{array}\right. 
	\end{equation}

Let $K_{\alpha}\neq\{0\}$ be a nonempty pseudo-face of $K$, given by 
$$K_\alpha=\big\{x\in {\mathbb R}^n:C_{i}x=0\ 
\forall i\in\alpha,\ C_{i}x< 0\ \forall i\in I\setminus\alpha\big\},$$ 
where $C_{i}$ is the $i$-th row of $C$. Thus, $X_d$ is nonzero on this pseudo-face.	We consider the function
$$\Phi_{\alpha}:\R^{n\times \rho}\times K_{\alpha}\times \R_+^{|\alpha|} \to \R^{n+|\alpha|},$$
which is defined by 
$$\Phi_\alpha(B,x,\lambda_\alpha)=\Big( BX_d+\displaystyle\sum_{i\in \alpha}\lambda_i C_i, C_\alpha x\Big),$$
	where
$C_\alpha x=(C_{i_1}x,\ldots,C_{i_{|\alpha|}x}), i_j\in\alpha.$ Clearly, $\Phi_\alpha$ is a smooth
semi-algebraic function.
	The Jacobian matrix 
	of $\Phi_\alpha$ is given by
$$D\Phi_\alpha=\left[ \begin{array}{c|c|c}
D_{B}(BX_d) \; & \ * \ &\  C_\alpha^T   \\  
\hline
0_{|\alpha|\times \rho}&\ \ \ C_\alpha \ \ \ & 0_{|\alpha|\times |\alpha|} \\ 
\end{array}\right],$$
where $0_{u\times v}$ is the zero $u\times v$--matrix. Here the $n\times \rho$--matrix  $D_{B}(BX_d)$  is given by 
	$$D_{B}(BX_d)=\begin{bmatrix}
	X_d	&0_{1\times \rho}  & \cdots &0_{1\times \rho}  \\ 
	0_{1\times \rho} 	&X_d	  & \cdots &0_{1\times \rho} \\ 
	\vdots & \vdots & \ddots & \vdots \\ 
	0_{1\times \rho} 	&0_{1\times \rho}   & \cdots & X_d	
	\end{bmatrix}.$$
	
	Since $X_d$ is nonzero on $K_{\alpha}$, the rank of $D_{B}(BX_d)$ is $n$. 	By our assumptions, the rank of the matrix $D\Phi_\alpha$ is $n+|\alpha|$ for all $x\in K_\alpha$. Therefore,	$0\in \R^{n+|\alpha|}$ is a regular value of $\Phi_\alpha$. The Sard theorem with parameter \cite[Theorem 2.4]{DHP16} says that there exists a generic semi-algebraic set  $\Sa_{\alpha}\subset \R^{n\times\rho}$,  such that if $B\in \Sa_{\alpha}$ then $0$ is a regular value of the map
	$$\Phi_{\alpha,B}:K_{\alpha}\times \R^{|\alpha|} \to \R^{n+|\alpha|}, \ \Phi_{\alpha,B}(x,\lambda_\alpha) =\Phi_\alpha(B,x,\lambda_\alpha).$$
According to the regular level set theorem \cite[Theorem 9.9]{Loring_2010}, if  $\Phi^{-1}_{\alpha,B}(0)$ 
	is nonempty then it is a $0-$dimensional semi-algebraic set. It follows that $\Phi^{-1}_{\alpha,B}(0)$ is a finite set. Moreover, from \eqref{KKT}, one has
	$$\Gamma_K(B)\cap K_{\alpha}=\pi(\Phi^{-1}_{\alpha,B}(0)),$$
	where $\pi$ is the projection $\R^{n+|\alpha|} \to \R^n$ defined by $\pi(x,\lambda_{\alpha}) = x$. Therefore, $\Gamma_K(B)\cap K_{\alpha}$ is a finite set.
	
Since $0\in \Gamma_K(B)$, $\Gamma_K(B)\cap K_{\alpha}=\{0\}$ if $K_{\alpha}=\{0\}$. By the finite decomposition \eqref{Sol_decom}, 
$$\Gamma_K(B)=\bigcup_{\alpha\subset I}\Gamma_K(B)\cap K_{\alpha}$$
  is a finite set. Taking $\Sa=\cap_{\alpha\subset I}\Sa_{\alpha},$ 
we see that $\Gamma_K(B)$ has finite points for any $B\in\Sa$. Recall that $\Gamma_K(B)$ is a closed cone which contains $0$. Hence, $\Gamma_K(B)=\{0\}$ for all $B$ in $\Sa$. 
	
	The proof is complete. \qed
\end{proof}

Consider the isomorphism $\varphi:\R^{n\times m} \to \Po_d$ defined by $\varphi(A)=AX$, where $X,A$ are as in \eqref{monomials} and \eqref{A}, respectively. A set
$\Sa$ is generic in $\R^{n\times m}$ iff $\varphi(\Sa)$ is generic in $\Po_d$.

\begin{remark}%\label{rk:Kinf}
Suppose that the constraint $K$ is given by \eqref{K}. For each $i\in I$, $j\in J$, one denotes
$K_i=\left\lbrace x \in \R^n : g_i(x)\leq 0\right\rbrace$ and $K'_j=\left\lbrace x \in \R^n : h_j(x)= 0\right\rbrace$. Clearly,	$K^{\infty}_i,K'^{\infty}_j$ are polyhedral convex cone (see \cite[p.39]{BK2002}). Since $$K=\big(\bigcap_{i\in I}K_i\big)\bigcap(\bigcap_{j\in J}K'_j\big),$$ it follows from \cite[Proposition 2.1.9]{AusTeb2003} that
$$K^{\infty}=\big(\bigcap_{i\in I}K^{\infty}_i\big) \bigcap \big(\bigcap_{j\in J}K'^{\infty}_j\big).$$
Thus, $K^{\infty}$ is a nonempty polyhedral convex cone.
\end{remark}

\begin{theorem}%\label{generic_2} 
Suppose that $K$ is given by \eqref{K} and $K^{\infty}$ is given by  \eqref{K_0}. If the $\rm LICQ$ holds on $K^{\infty}$, then the set $\Ri_0(K,d)$ is generic in $\Po_d$.
\end{theorem}
\begin{proof} Assume that $\rm LICQ$ holds on $K^{\infty}$. According to Lemma \ref{generic_1}, there exists a generic set $\A\subset \Hd_d$ such that the cone $\Sol(K^{\infty},H)$ is trivial, for all $H\in \A$.
From the direct sum $\Po_d=\Hd_d\oplus\Po_{d-1}$, where
$$\Hd_d\oplus\Po_{d-1}=\{H+Q \; | \; H\in\Hd_d, Q\in \Po_{d-1}\},$$ 
we can assert that $\A\oplus\Po_{d-1}$ 
is generic in $\Po_d$ and $(K,P)$ is an $R_0$-pair for any $P\in \A\oplus\Po_{d-1}$. It follows that $\Ri_0(K,d)$ is also generic in $\Po_d$. The proof is complete. \qed
\end{proof}

\subsection{Genericity of the finite-valuedness}

Recall that a set-valued map $\Psi$ is \textit{finite-valued} on $X$ if the cardinality of $\Psi(x)$ is finite, i.e. $|\Psi(x)|<+\infty$ for every $x\in X$. 

The finite-valuedness of $\Sol_{P}$ on a generic set of $\R^n$  was announced and proved in \cite[Theorem 3.2]{LLP2018}. The following theorem mention the genericity of the finite-valuedness of the map $\Sol$.

\begin{theorem}\label{generic_3}
Assume that $K$ is a semi-algebraic set given by \eqref{K} and the $\rm LICQ$ holds on $K$. Then, the map $\Sol$ is finite-valued on a generic set of $\Po_d$.
\end{theorem}

\begin{proof} For $Q\in\Po_{d}$, one has $Q(x)=AX$, where $X$ and $A$ are given by \eqref{monomials} and \eqref{A}, respectively. For each nonempty pseudo-face $K_{\alpha}$ of $K$, we consider the function
	$$\Phi_\alpha:\R^{n\times m}\times K_{\alpha}\times \R^{|\alpha|+|J|} \to \R^{n+|\alpha|+|J|},$$
	which is defined by 
	$$\Phi_\alpha(A,x,\lambda_\alpha,\mu)=\Big(AX+\displaystyle\sum_{i\in \alpha}\lambda_i\nabla g_i(x)+\displaystyle\sum_{j\in J}\mu_i\nabla h_j(x), g_\alpha(x),h(x)\Big),$$
	where
	$g_\alpha(x)=(g_{i_1}(x),\dots,g_{i_{|\alpha|}}(x)), i_j\in\alpha$.	
An easy computation shows that the Jacobian matrix $D_{A}(AX)$ is the following $n\times m-$matrix
	\begin{equation}\label{Da_P}
	D_{A}(AX)=\left[ \begin{array}{cccc}
	X&0_{1\times m}&\cdots & 0_{1\times m} \\
	0_{1\times m}&X&\cdots & 0_{1\times m}\\
	\vdots&\vdots&\ddots& \vdots \\
	0_{1\times m}&0_{1\times m}&\cdots & X\\
	\end{array}\right].
	\end{equation}
From \eqref{Da_P}, the rank of $D_{A}(AX)$ is $n$.
Since $K$ has the $\rm LICQ$ property, the rank of the Jacobian matrix $D\Phi_{\alpha}$ is $n+|\alpha|+|J|$ for all $x\in K_\alpha$.

Repeating the argument of the proof of Lemma \ref{generic_1} (by using the Sard theorem with parameter and the regular level set theorem), one can assert that there exists a semi-algebraic set $\Sa$, which is  generic in $\R^{n\times m}$, such that $\Sol(K,AX)$  is a finite set for every $A\in\Sa$. Hence, we can conclude that there is a generic set in $\Po_d$ such that $\Sol$ is finite-valued on it.  \qed
\end{proof}

\begin{remark} Suppose that the assumptions in Theorem \ref{generic_3} are satisfied. We can show that if the  map $\Sol$  is lower semicontinous at $P$, then $\Sol(K,P)$ has finitely many points. Hence, if $\dim\Sol(K,P)\geq 1$, then $\Sol$ is not lower semicontinuous at $P$.
\end{remark}

\begin{acknowledgements}
The author would like to thank Prof. Nguyen Dong Yen,  Dr. Pradeep Kumar Sharma, Dr. Vu Thi Huong, and the first anonymous referee for their corrections. The author is indebted to the third anonymous referee for numerous valuable comments and suggestions.
\end{acknowledgements}

%References%

\end{document}